\documentclass[12pt]{article}
\usepackage{amsmath,amssymb}
\usepackage{amsthm}
\usepackage{epsfig}
\usepackage{graphics}
\usepackage{graphicx}
\usepackage{colordvi}
\usepackage{mathrsfs} 
\usepackage{stmaryrd}
\usepackage{geometry}
\usepackage{dsfont}
\usepackage{tikz}
\usetikzlibrary{automata}
\usepackage{url}
\usepackage
{hyperref}
\hypersetup{
    colorlinks=true,
    linkcolor=blue!75!black,
    filecolor=blue!75!black,      
    urlcolor=blue!75!black, 
    citecolor=blue!75!black,
}
\usepackage[textsize=small]{todonotes}
\usepackage[margin=1cm, size=small]{caption}
\usepackage{multirow}
\allowdisplaybreaks[4]
\oddsidemargin
-.5in \evensidemargin -.5in\marginparwidth 0.4in
\usepackage{color,soul}
\usepackage{xcolor}

\usepackage{etoolbox}
\patchcmd{\thebibliography}{\section*}{\section}{}{}

\newcommand{\1}{\mathds 1}

\newcommand{\F}{\mathscr F}

\newcommand{\vep}{{\varepsilon}}

\newcommand{\ms}{\mathscr}
\renewcommand{\P}{{\mathbb P}}
\newcommand{\E}{{\mathbb E}}

\newcommand{\R}{{\Bbb R}}
\newcommand{\Id}{{\rm Id}}

\newcommand{\tr}{{\sf tr}}
\newcommand{\defeq}{{\stackrel{\rm def}{=}}}
\newcommand{\M}{{M}}
\renewenvironment{proof}[1][\proofname]{\noindent {\bfseries #1.}\;}{\hfill\ensuremath{\blacksquare}\\}
\renewcommand{\d}{{\mathrm d}}
\newcommand{\e}{{\mathrm e}}
\usepackage{currfile}
\usepackage{fancyhdr}
\newtheoremstyle{slantthm}{10pt}{10pt}{\slshape}{}{\bfseries}{}{.5em}{\thmname{#1}\thmnumber{ #2}\thmnote{ (#3)}.}
\newtheoremstyle{slantrmk}{10pt}{10pt}{\rmfamily}{}{\bfseries}{}{.5em}{\thmname{#1}\thmnumber{ #2}\thmnote{ (#3)}.}

\begin{document}
\theoremstyle{slantthm}
\newtheorem{thm}{Theorem}[section]
\newtheorem{prop}[thm]{Proposition}
\newtheorem{lem}[thm]{Lemma}
\newtheorem{cor}[thm]{Corollary}
\newtheorem{defi}[thm]{Definition}
\newtheorem{disc}[thm]{Discussion}
\newtheorem*{nota}{Notation}
\newtheorem{conj}[thm]{Conjecture}
\newtheorem*{mr}{Main Result (Informal statement)}

\theoremstyle{slantrmk}
\newtheorem{ass}[thm]{Assumption}
\newtheorem{rmk}[thm]{Remark}
\newtheorem{eg}[thm]{Example}
\newtheorem{que}[thm]{Question}
\numberwithin{equation}{section}
\newtheorem{quest}[thm]{Quest}
\newtheorem{prob}[thm]{Problem}

\title{\bf Precise asymptotics of some meeting times arising from the voter model on large random regular graphs\thanks{Support from the Simons Foundation before the author's present position and from the Natural Science and Engineering Research Council of Canada is gratefully acknowledged.
}}

\author{Yu-Ting Chen\footnote{Department of Mathematics and Statistics, University of Victoria, British Columbia, Canada.} \footnote{Email: \url{chenyuting@uvic.ca}}}

\date{\today}

\maketitle
\abstract{We consider two independent stationary random walks on large random regular graphs of degree $k\geq 3$ with $N$ vertices. On these graphs, the exponential approximations of the meeting times are known to follow from existing methods and form a basis for the voter model's diffusion approximations. The main result of this note improves the exponential approximations to an explicit form such that the first moments are asymptotically equivalent to  $N(k-1)/[2(k-2)]$.\medskip 

\medskip

 
}

\section{Introduction}\label{sec:intro}
This note is concerned with the meeting time $M=\inf\{t\geq 0;X_t=Y_t\}$ of i.i.d. continuous-time, rate-$1$ irreducible Markov chains $X$ and $Y$ defined on a large finite set; the chains are subject to the stationary initial conditions. These basic stopping times $M$'s arise in a series of studies of diffusion approximations of the voter model and some closely related interacting particle systems \cite{Cox_1989, Oliveira, CCC, CC, C}. In this context, the first moments of the meeting times are the time changes for diffusion approximations of the particle systems, but they encode the underlying Markov transition kernels implicitly. Due to this connection, various non-rigorous, explicit approximations from the physics literature (cf. \cite{SR,VE}) for the particle systems  may be translated to precise asymptotics of the meeting times or related objects, but very few are mathematically proven. 

Our primary interest in this note is to establish a new example for the precise asymptotics of the first moments of the meeting times. We consider the meeting times of random walks on large random regular $k$-graphs, for any fixed degree $k\geq 3$. (See Section~\ref{sec:rrg} for the definition of these graphs.) The possibility of proving the precise asymptotics is suggested by  a result in \cite{OHLN}  for an evolutionary game model that can be identified as a weak perturbation of the voter model. In terms of time changes for diffusion approximations, a comparison of \cite{OHLN} with the established mathematical result [see also Remark~\ref{rmk:rrg} (2)] shows that the following limit should hold:
\begin{align}\label{rrg}
\E_G[M]\sim \frac{N(k-1)}{2(k-2)}\quad \mbox{as }N\to\infty.
\end{align}
Here, $\E_G[M]$ is defined on the random $k$-regular graph $G$ with $N$ vertices, and $a_N\sim b_N$ if $a_N/b_N\to 1$. See \cite{C}, especially Section 4.3 there, for the background of (\ref{rrg}). 

Relative to the precise asymptotics of $\E_G[M]$, the convergence of the normalized times $M/\E_G[M]$ is known. It is just one particular case of the exponential approximations of hitting times of small sets by stationary Markov chains which are valid under mild conditions \cite{Aldous:AE,AB,AF}. (Bounds of the expected hitting times are also obtained there.) Specialized to the meeting times on large random regular graphs, the conditions in \cite{Aldous:AE,AB,AF} essentially require that the first moments $\E_G[M]$ grow more rapidly than the times for the associated bivariate Markov chains $(X_t,Y_t)$ to mix. Then the limit of $M/\E_G[M]$ is the exponential random variable with a mean one. See the proofs in \cite[Section~6]{CCC} and \cite[Section~4.3]{C} for verifications of these conditions in general. 

In the different direction of representing distributions explicitly on finite sets other than random regular graphs, the meeting times are known to be reducible to hitting times of points. The weakest exact symmetry known to us for this reduction is  $\P(X_t=x|X_0=x)$  independent of $x$ for all $t$  \cite[Section~14.2]{AF}. In this case, the first moments of the meeting times (again by stationary chains) can also be expressed explicitly as sums of elementary functions of \emph{all} the eigenvalues, known as Kemeny's constant. See Section~\ref{sec:spec} and the eigentime identity in \cite{AF}. Hence, the exact symmetry gives details of the meeting time distributions much more than those from the exponential approximations by quite different methods. In particular, it allows for explicit formulas of the first moments.

In terms of this background, it is unclear whether the general methods in \cite{Aldous:AE,AB,AF} for exponential approximations are enough to obtain the explicit asymptotics in \eqref{rrg}. See also the end of Section~\ref{sec:rrg}. By extending the method for proving Kemeny's constant, the main result establishes \eqref{rrg} and improves the exponential approximations accordingly. 

\begin{mr}
\sl Given a fixed integer $k\geq 3$, let $G$ be the $k$-random regular graphs on $N$ vertices and $M$ the meeting time of two independent stationary random walks on $G$. Then, in the sense of convergence in distribution and convergence of all moments, $M/N\to(k-1)\mathbf e/[2(k-2)]$ as $N\to\infty$, where  $\mathbf e$ denotes an exponential random variable with $\E[\mathbf e]=1$.
\end{mr}

The proof of this result begins with the well-known property that the infinite $k$-regular tree is the limit of  large random $k$-regular graphs \cite{McKay, BS}. The symmetry of this infinite graph is applied in several crucial ways to extend by approximations a spectral method in the spirit of proving Kemeny's constant. After all,  although the exact symmetry leading to Kemeny's constant is violated on random regular graphs, the asymmetry is mild, and an analog of the reduction to hitting times of points mentioned above still applies to the infinite tree. This property suggests that the tree is a reference geometry for approximations. Nevertheless, the critical issue is whether the mild asymmetry breaking the exact relation to Kemeny's constant also has \emph{comparable mild effects} on the meeting time distributions. A key step of the proof verifies this relation in general by showing that the stationary distributions dominate the meeting location distributions. See Proposition~\ref{prop:recursion}. The precise limit in \eqref{rrg} is then determined by applying the Kesten--McKay law for the spectral measure of the $k$-regular tree. Here, we consider an analog of the eigentime identity to extend Kemeny's constant to the infinite tree. \smallskip 

\noindent {\bf Organization of this {note}.} The proof of the main theorem is given in Section~\ref{sec:rrg}. Sections~\ref{sec:spec}, \ref{sec:prob} and~\ref{sec:appendix} present some auxiliary results for general finite Markov chains.\smallskip

\section{Spectral representation}\label{sec:spec}
We begin with the basic setup of Markov chains. Let $Q$ be an irreducible, reversible transition kernel defined on a finite set $E$ with $\#E=N$. Assume that $Q$ has a zero trace: $\sum_x Q(x,x)=0$. Let $\{X^x,Y^x; x\in E\}$ be a family of independent (rate-$1$) $(E,Q)$-Markov chains with $X^x_0=Y^x_0=x$.  By functional calculus (Section~\ref{sec:appendix}), $\P(X^x_t=z)=\e^{t(Q-1)}(x,z)$ for all $x,z$. Since $E$ is finite and the irreducibility of $Q$ implies the irreducibility of a product of two $Q$-chains, a standard result of Markov chains ensures that $(X^x,Y^y)$ hits the diagonal $\{(z,z);z\in E\}$ a.s., that is, $X^x$ and $Y^y$ meet a.s. Hence, $M_{x,y}=\inf\{t\geq 0;X_t^x=Y^y_t\}$ is finite a.s. With $\pi$ denoting by the unique stationary distribution of $Q$, we write $M=M_{U,U'}$ for $(U,U')$ distributed as $\pi\otimes \pi$ and independent of all $X^x$ and $Y^y$. 

The following lemma is the starting point of this note to study the distribution of $M$. It is followed by a classical connection to the spectrum of $Q$ [see (\ref{norm_trace})] which the method in the next section aims to extend. 

\begin{lem}
{\sl
For all $\lambda\in (0,\infty)$, 
\begin{align}\label{G-rec}
\frac{1}{\lambda}\sum_{z\in E}\pi(z)^2=\E\big[\e^{-\lambda \M}G_{\lambda}\big(X^U_\M,X^{U}_\M\big)\big],
\end{align}
where $G_\lambda$ is the Green function defined by  
\begin{align}\label{def:Glambda}
G_\lambda (x,y)\stackrel{\rm def}{=}\int_0^\infty \e^{-\lambda t}\P(X^{x}_t=Y^{y}_t)\d t,\quad x,y\in E,\;\lambda\in (0,\infty).
\end{align}
}
\end{lem}
\begin{proof}
We have
\[
\E\left[\int_0^\infty \e^{-\lambda t}\1_{\{X^U_t=Y^{U'}_t\}}\d t\right]=\E\left[\int_M^\infty \e^{-\lambda t}\1_{\{X^U_t=Y^{U'}_t\}}\d t\right]
=\E\big[\e^{-\lambda \M}G_{\lambda}\big(X^U_\M,X^{U}_\M\big)\big]
\]
by the strong Markov property of $(X^{U},Y^{U'})$ at $\M$. Also, independence and stationarity imply that the right-hand is reduced to $\lambda^{-1}\sum_{z\in E}\pi(z)^2$. We have proved \eqref{G-rec}. 
\end{proof}

Define an inner product for functions on $E$ by
\begin{align}\label{ip}
\textcolor{black}{\langle f,g\rangle\;\defeq\;\sum_{y\in E}f(y)g(y)},
\end{align}
and denote by $\delta_x$ the delta function at $x\in E$: $\delta_x(y)$ is $1$ if $y=x$ and is zero otherwise. In the case that $Q$ is symmetric, the stationary distribution $\pi$ is uniform on $E$, and $\e^{t(Q-1)}(y,z)=\e^{t(Q-1)}(z,y)$ so that 
\begin{align}
 G_{\lambda}(x,y)
=&\sum_{z\in E}\int_0^\infty \e^{-\lambda t}\e^{t(Q-1)}(x,z)\e^{t(Q-1)}(z,y)\d t\notag\\
=&\int_0^\infty \e^{-\lambda t}\left\langle \delta_x,\e^{2t(Q-1)}\delta_y\right\rangle \d t\hspace{2cm} \mbox{(summing over $z$)}\notag\\
=&\left\langle \delta_x,\frac{1}{\lambda+2(1-Q)}\delta_y\right\rangle\label{Gxy0}\hspace{2.85cm} \mbox{(functional calculus)}.
\end{align}
If, moreover, $G_{\lambda}(x,x)$ is independent of $x$,  then it equals $N^{-1}\sum_{y\in E}G_{\lambda}(y,y)$. In this case, a division of both sides of  (\ref{G-rec}) by this normalized sum yields
\begin{align}\label{norm_trace}
\E\big[\e^{-\lambda M}\big]=\frac{(\lambda N) ^{-1}}{N^{-1}\sum_{x\in E}G_{\lambda}(x,x)}
=\frac{(\lambda N)^{-1}}{N^{-1}\tr\left(\frac{1}{\lambda+2(1-Q)}\right)}
\end{align}
by (\ref{Gxy0}). In other words, the distribution of $M$ can be represented explicitly by the normalized spectral measure $B\mapsto N^{-1}\tr\big(\1_B(Q)\big)$.

\begin{eg}\label{eg:tori}
On a discrete torus of dimension $d\geq 3$, \eqref{norm_trace} applies since $Q$ is symmetric and  the constancy of $Q^\ell(x,x)$ in $x$ for all $\ell\in \Bbb  Z_+$ holds. Indeed, this constancy is equivalent to the constancy of $G_\lambda(x,x)$ in $x$ for all $\lambda\in (0,\infty)$ since  $G_\lambda(x,x)$ is the Laplace transform of $t\mapsto Q_t(x,x)$ and $\e^tQ_t(x,x)$ is the generating function of $\ell\mapsto Q^{\ell}(x,x)/\ell!$ in $t$ [see (\ref{Qt:sum})]. In this case, explicit asymptotic results of the Laplace transforms can be obtained from the known eigenvalues of the discrete-time random walks (cf. \cite[Section~12.3.1 and Lemma~12.11]{LPW}). The scaling of $\lambda$  for the asymptotics is the straightforward $1/N$. Observe that this spectral method can be seen as a different facet of  the proof of \cite[Theorem~7]{Cox_1989}, although that proof originally uses the characteristic functions of the random walks to represent $G_\lambda(x,y)$. See also \cite[Section~6]{JLT}. \hfill $\blacksquare$
\end{eg}

We turn to the case of large random regular graphs in the next section and resume the setup of general Markov chains afterward. 

\section{Asymptotics on large random regular graphs}\label{sec:rrg}
In this section, we derive the asymptotic distribution of the meeting time $M$ on a large random regular graph. For the basic terminology of graph theory used below, we refer the reader to \cite{Bollobas:GT} for the details.  

The random regular graphs are defined as follows. For a fixed integer $k\geq 3$, we choose a sequence $\{N_n\}$ of positive integers such that $N_n\to\infty$ and $k$-regular graphs (without loops and multiple edges)  on $N_n$ vertices exist. The choice of these integers $N_n$ follows from an application of the Erd\H{o}s--Gallai necessary and sufficient condition (cf. \cite{TVW}), which requires that $kN_n$ be even and $k\leq N_n-1$ in the present case. Then  the random regular graph on $N_n$ vertices is the graph $G_n$ uniformly chosen from the set of $k$-regular graphs with $N_n$ vertices. We assume that the randomness defining the graphs is collectively subject to the probability $\mathbf P$ and expectation $\mathbf E$, as opposed to the quenched probability $\P^{(n)}$ and quenched expectation $\E^{(n)}$ for random walks on $G_n$'s.  

The random walk on $G_n$ has a symmetric transition kernel $Q^{(n)}$ such that $Q^{(n)}(x,y)=1/k$ whenever there is an edge between $x$ and $y$, and $Q^{(n)}(x,y)=0$ otherwise. One stationary distribution $\pi^{(n)}$ of $Q^{(n)}$ is given by the uniform distribution.\vspace{.1cm}

\noindent \hypertarget{P1}{{\bf (P1)}} $\pi^{(n)}(x)\equiv 1/N_n$.\vspace{.1cm}

On these graphs,  the expression (\ref{Gxy0}) for $G_{\lambda}(x,y)$ remains valid since $Q$ is symmetric. We do not know if the trace formula in (\ref{norm_trace}) still applies since the $\bf P$-probability that $Q^{(n),\ell}(x,x)$ is independent of $x$ for any $\ell\geq 1$ does not tend to one as $n\to\infty$, and so, by the argument in Example~\ref{eg:tori}, the constancy of $G_\lambda(x,x)$ in $x$ for all $\lambda$ breaks down. 

This lack of constancy of $Q^{(n),\ell}(x,x)$ can be seen as follows. Recall that a cycle is a sequence of edges $(x_0,x_1),(x_1,x_2),\cdots,(x_{r-1},x_r)$ defined by vertices $x_0,x_1,\cdots,x_r$ such that $x_0=x_r$ and $x_0,x_2,\cdots,x_{r-1}$ are distinct. It is known that the number $C_n(r)$ of cycles of length $r$ in $G_n$ converges in distribution to a Poisson random variable with mean $(k-1)^r/(2r)$ for every $r\geq 3$. See \cite[Section~2.4]{Bollobas}. Hence, for example, the $\mathbf P$-probability to find two distinct vertices $x_n,y_n$ with $Q^{(n),3}(x_n,x_n)=0$ and $Q^{(n),3}(y_n,y_n)>0$ tends to one. (Here, $Q^{(n),\ell}$ is the $\ell$-th step transition probability of $Q^{(n)}$.) In the rest of this section, we show an extension of (\ref{Gxy0}), using the additional properties (\hyperlink{P2}{P2}) and (\hyperlink{P3}{P3}) of the random regular graphs introduced below.  

First, the order-$1$ limiting law of $C_n(r)$ implies a locally tree-like property: As $n\to\infty$,  $rC_n(r)/N_n$ converges to zero in probability for every  fixed $r$, whereas $rC_n(r)$ is an easy bound for the number of vertices which can be passed through by an $r$-cycle. Hence, (\hyperlink{P1}{P1}) implies the following property. Here, $\xrightarrow[n\to\infty]{\mathbf P}$ denotes convergence in $\mathbf P$-probability as $n\to\infty$.  \vspace{.1cm}

\noindent  \hypertarget{P2}{{\bf (P2)}} For every $\ell\in \Bbb  Z_+$, we can find a constant $Q^{(\infty),\ell}$ such that 
\begin{align}\label{lc}
\pi^{(n)}\big\{x\in E_n;Q^{(n),\ell}(x,x)\neq Q^{(\infty),\ell}\big\}\xrightarrow[n\to\infty]{{\mathbf P}}0,
\end{align}
where $\pi^{(n)}$ is given by (\hyperlink{P1}{P1}).
\vspace{.1cm}

The transition probability $Q^{(n)}$ on $G_n$ is the $k^{-1}$ multiple of the adjacency matrix. Hence, by the locally tree-like property mentioned above and the spatial homogeneity of the infinite tree, $Q^{(\infty),\ell}=k^{-\ell}\cdot \#\{\mbox{$x$-$x$ walks of length $\ell$}\}$ for any vertex $x$. (As above, see \cite{Bollobas:GT} for the precise definition of the terminology from graph theory.) McKay 
\cite{McKay} shows
\begin{align}\label{qinftyell}
k^\ell Q^{(\infty),\ell}=\int_\R q^\ell \mu_k(\d q),
\end{align}
where the measure $\mu_k$  is now often known as  the Kesten--McKay law:
\[
\mu_k(\d q)=\1_{(-2\sqrt{k-1},2\sqrt{k-1})}(q)\frac{k\sqrt{4(k-1)-q^2}}{2\pi(k^2-q^2)}\d q,\quad q\in\ \R.
\]
By (\ref{qinftyell}), $\mu_k$ is the spectral measure of the adjacency matrix of the infinite tree.

For the following proof, we  only need the explicit form of $\sum_{\ell=0}^\infty Q^{(\infty),\ell}$. As a particular case of \cite[(16.20) and (16.21)]{AW} where the adjacency matrix is viewed as an operator acting on square-summable functions, we have
\[ 
\int_{\R}\frac{1}{k-q}\mu_k(\d q)=\frac{1}{k-k\Gamma}\quad  \mbox{ for $\Gamma$ satisfying }\quad 
\sqrt{k-1}\Gamma=\frac{k}{2\sqrt{k-1}}-\sqrt{\frac{k^2}{4(k-1)}-1}.
\]
It follows that
\begin{align}\label{sum:Qell}
\sum_{\ell=0}^\infty Q^{(\infty),\ell}
=\sum_{\ell=0}^\infty \int_{\R}\left(\frac{q}{k}\right)^\ell \mu_k(\d q)=\int_\R \frac{1}{1-q/k}\mu_k(\d q)
=\frac{k-1}{k-2}.
\end{align}
Since $\mu_k$ is the spectral measure of the adjacency matrix of the $k$-regular tree, the integral $\int 1/(1-q/k)\mu_k(\d q)$ is an extension of the spectral formula for Kemeny's constant \cite[Proposition~3.13]{AF}. For this reason, the main theorem below may be seen as an extension of a basic identity between hitting times and meeting times on graphs with good symmetry as mentioned in the introduction (cf. \cite[Proposition~14.5]{AF}). 

The last property is for the spectral gaps of the random regular graphs. See \cite{Friedman_2008,Bordenave_2015}.  \medskip

\noindent  \hypertarget{P3}{{\bf (P3)}} Write $\lambda^{(n)}_{N_n}\leq \lambda^{(n)}_{N_n-1}\leq \cdots \leq \lambda^{(n)}_{1}=1$ for the eigenvalues of $Q^{(n)}$. For some ${\bf g}\in (0,1)$, the events $\Lambda_n\;\defeq\;\big\{\lambda^{(n)}_{r}\subseteq [-1+{\bf g},1-{\bf g}],\;\forall\;2\leq r\leq N_n\big\}$ satisfy $\mathbf P(\Lambda_n)\to 1$. \medskip

On $\Lambda_n$, there is only one connected component of $G_n$ since $\lambda^{(n)}_2< 1$ and $kQ^{(n)}$ is the adjacency matrix of $G_n$ \cite[Lemma~1.7 (iv)]{Chung}, and so $Q^{(n)}$ is irreducible. The uniform distribution in (\hyperlink{P1}{P1}) is thus the unique stationary distribution of $Q^{(n)}$  \cite[Proposition~1.14]{LPW}.  

\begin{rmk}\label{rmk:rrg}
(1) (\hyperlink{P3}{P3}) is equivalent to the property that every subsequence of $\{G_{n_i}\}$ contains a further subsequence $\{G_{n_{i_j}}\}$ such that $\mathbf P$-a.s., for some random integer $j_0\geq 1$, all the eigenvalues but the first one of the random walk on $G_{n_{i_j}}$ are contained in $ [-1+\mathbf g,1-\mathbf g]$ for all $j\geq j_0$. 
 
To see this connection, note that by the Borel--Cantelli lemma, $\mathbf P(\Lambda_n^\complement)\to 0$ implies that every subsequence $\{G_{n_i}\}$ contains a further subsequence $\{G_{n_{i_j}}\}$ such that $\mathbf P\big(\limsup_{j\to\infty}\Lambda_{n_{i_j}}^\complement\big)=0$, or equivalently $\mathbf P\big(\liminf_{j\to\infty}\Lambda_{n_{i_j}}\big)=1$. The converse is implied by Fatou's lemma since $\mathbf P\big(\liminf_{j\to\infty}\Lambda_{n_{i_j}}\big)=1$ gives $\lim_{j\to\infty}\mathbf P(\Lambda_{n_{i_j}})=1$.\medskip
 
\noindent (2) In \cite[Section~4.3]{C}, the convergence of some weak perturbation of the voter model based on the random regular graphs is obtained. The proof applies the property that $\mathbf P$-a.s., the second eigenvalue $\lambda^{(n)}_2$ is bounded away from $1$ for all large $n$. 

This property of the second eigenvalues is not the same as the property in (1). We do not know if the former holds or not. Hence, to be precise, given this fact for the context of $\mathbf P$-a.s. convergence, the statement of the convergence result in \cite[Section~4.3]{C} should be changed to the one that passes the limit along an appropriate subsequence of any given subsequence of $\{G_n\}$. See also the first statement of Theorem~\ref{thm:MUU} below. 
\hfill $\blacksquare$
\end{rmk}

Equipped with (\hyperlink{P1}{P1})--(\hyperlink{P3}{P3}) specified above, we proceed to the proof of the explicit asymptotics of the meeting times $M$ and their first moments. On $\Lambda_n$, $Q^{(n)}$ satisfies the assumptions at the beginning of Section~\ref{sec:spec}. Also, note that $M=+\infty$ with positive probability if the underlying graph is not connected, although what happens on $\Lambda_n^\complement$ is not important in the limit. Considering (\ref{Gxy0}), we extend the nontrivial contribution in $G_\lambda(x,y)$ under $Q^{(n)}$ to the case $\lambda=0$ by setting
\begin{align}\label{def:G<}
G_\lambda^{<}(x,y)\;\defeq\;\left\langle \delta_x,\frac{1}{\lambda+2(1-Q^{(n)})}\1_{[-1,1)}(Q^{(n)})\delta_y\right\rangle,\quad \lambda\in [0,\infty),\;\mbox{ on $\Lambda_n$}.
\end{align}
For convenience, we set $G^<_\lambda(x,y)\equiv \sum_{\ell=0}^\infty 2^\ell Q^{(\infty),\ell}/(\lambda+2)^{\ell+1}$ on $\Lambda_n^\complement$. 

\begin{lem}\label{lem:main}
\sl 
For every $\lambda\in (0,\infty)$,
\begin{align}\label{convP}
&\E^{(n)}\big[\e^{-\lambda \M/N_n}G^<_{\lambda/N_n}\big(X^U_\M,X^U_\M\big)\big]-
\E^{(n)}\big[\e^{-\lambda \M/N_n}\big]\sum_{\ell=0}^\infty \frac{2^\ell Q^{(\infty),\ell}}{(\lambda/N_n+2)^{\ell+1}}\xrightarrow[n\to\infty]{\mathbf P}0.
\end{align}
\end{lem}

\begin{proof}Write 
$\E^{(n)}\big[\e^{-\lambda \M/N_n};X^U_M=x\big]$ for $\E^{(n)}\big[\e^{-\lambda \M/N_n}\1_{\{X^U_M=x\}}\big]$. Here and in (\ref{convP}), we use the convention  that $\e^{-\infty}=0$ when $M=+\infty$.

For every fixed $n\in \Bbb  N$, the difference in (\ref{convP}) is zero on $\Lambda_n^\complement$. On $\Lambda_n$, for arbitrary odd integer $L\geq 1$, it holds that 
\begin{align}
&\quad \E^{(n)}\left[\e^{-\lambda \M/N_n}G^<_{\lambda/N_n}\big(X^U_\M,X^U_\M\big)\right]\notag\\
&=\sum_{x\in E_n}\E^{(n)}\big[\e^{-\lambda \M/N_n};X^U_M=x\big]
\left\langle \delta_x,\frac{1}{\lambda/N_n+2(1-Q^{(n)})}\1_{[-1,1)}(Q^{(n)})\delta_x\right\rangle\notag\\
&=\sum_{x\in E_n}\E^{(n)}\big[\e^{-\lambda \M/N_n};X^U_M=x\big]
\left\langle \delta_x,\sum_{\ell =0}^\infty \frac{2^\ell (Q^{(n)})^\ell }{(\lambda/N_n+2)^{\ell +1}}\1_{[-1,1)}(Q^{(n)})\delta_x\right\rangle\notag\\
&=\sum_{x\in E_n}\E^{(n)}\big[\e^{-\lambda \M/N_n};X^U_M=x\big]
\sum_{\ell =0}^L \frac{2^\ell \langle \delta_x, (Q^{(n)})^\ell \delta_x\rangle}{(\lambda/N_n+2)^{\ell +1}}\notag\\
&\quad -\sum_{x\in E_n}\E^{(n)}\big[\e^{-\lambda \M/N_n};X^U_M=x\big]
\sum_{\ell =0}^L \frac{2^\ell \langle \delta_x, \1_{\{1\}}(Q^{(n)})\delta_x\rangle}{(\lambda/N_n+2)^{\ell +1}}\notag\\
&\quad +\sum_{x\in E_n}\E^{(n)}\big[\e^{-\lambda \M/N_n};X^U_M=x\big]
\sum_{\ell =L+1}^\infty \frac{2^\ell \langle \delta_x, (Q^{(n)})^\ell \1_{[-1,1)}(Q^{(n)})\delta_x\rangle}{(\lambda/N_n+2)^{\ell +1}}\notag\\
&={\rm I}-{\rm II}+{\rm III}. \label{sum:1}
\end{align}
Note that {\rm III} is nonnegative by \eqref{def:fA} and the nonnegativity of $\sum_{\ell=L+1}^\infty \frac{2^\ell q^\ell}{(\lambda/N_n+2)^{\ell+1}}$ on $q\in [-1,1)$, since $L+1$ is even by the choice of $L$. By \eqref{def:fA} again, {\rm II} is plainly nonnegative. 

As a counterpart of \eqref{sum:1}, we write 
\begin{align}
&\quad \E^{(n)}\big[\e^{-\lambda \M/N_n}\big] \sum_{\ell=0}^\infty\frac{2^\ell Q^{(\infty),\ell} }{(\lambda/N_n+2)^{\ell +1}}\notag\\
&=\sum_{x\in E_n}\E^{(n)}\big[\e^{-\lambda \M/N_n};X^U_M=x\big]\sum_{\ell=0}^L\frac{2^\ell Q^{(\infty),\ell}}{(\ell/N_n+2)^{\ell+1}}+\E^{(n)}\big[\e^{-\lambda \M/N_n}\big]\sum_{\ell=L+1}^\infty\frac{2^\ell Q^{(\infty),\ell}}{(\lambda/N_n+2)^{\ell+1}}\notag\\
&={\rm I}'+{\rm III}',\label{sum:2}
\end{align}
where the trivial sum over $x\in E_n$ is for the convenience of the following argument. 

For any arbitrary $\vep>0$, we can choose $L$ large enough such that 
\begin{align}\label{def:L}
\sum_{\ell=L+1}^\infty (1-\mathbf g)^\ell\leq \frac{\vep}{4}
\quad\&\quad 
\sum_{\ell=L+1}^\infty Q^{(\infty),\ell}\leq \frac{\vep}{4} 
\end{align}
by (\hyperlink{P3}{P3}) and (\ref{sum:Qell}).  Now, we compare both sides of (\ref{sum:1}) and (\ref{sum:2}). It is enough to show all of the following equalities:
\begin{align}\label{cip}
\begin{split}
0&=\lim_{n\to\infty}\mathbf P\left(|{\rm I}-{\rm I}'|>\frac{\vep}{4},\Lambda_n\right)
=\lim_{n\to\infty}\mathbf P\left({\rm II}>\frac{\vep}{4},\Lambda_n\right),\\
0&=\mathbf P\left({\rm III}>\frac{\vep}{4},\Lambda_n\right)
=\mathbf P\left({\rm III}'>\frac{\vep}{4},\Lambda_n\right).
\end{split}
\end{align}
If all of these are proven, then along with the convergence $\mathbf P(\Lambda_n^\complement)\to 0$ from (\hyperlink{P3}{P3}), the probability that the absolute value of the difference in \eqref{convP} is greater than $\vep$ tends to zero as $n\to\infty$. The convergence in (\ref{convP}) thus follows since $\vep>0$ is arbitrary.

For ${\rm I}-{\rm I}'$, we write $\langle \delta_x,(Q^{(n)})^\ell\delta_x\rangle=Q^{(n),\ell}(x,x)$ and apply a simple bound proven later on for $\E^{(n)}[\e^{-\lambda M};X^U_M=\,\cdot\,]$ [see \eqref{dominate}] along with (\hyperlink{P1}{P1}). These steps give
\begin{align}
|{\rm I}-{\rm I}'|&\leq \sum_{\ell=0}^L\sum_{x\in E_n}\left(\frac{2+\lambda/N_n}{\lambda/N_n}\cdot \frac{1}{N_n^2}\right)\cdot \frac{2^\ell|Q^{(n),\ell}(x,x)-Q^{(\infty),\ell}|}{(\lambda/N_n+2)^{\ell+1}}\label{I-I}\\
&\leq \frac{1}{\lambda}\sum_{\ell=0}^L\int_{E_n}|Q^{(n),\ell}(x,x)-Q^{(\infty),\ell}| \pi^{(n)}(
\d x)\notag
\end{align}
by (\hyperlink{P1}{P1}) again. Since $Q^{(n),\ell}(x,x)$ and $Q^{(\infty),\ell}$ are all bounded by $1$, it follows from (\hyperlink{P2}{P2}) that the required convergence for ${\rm I}-{\rm I}'$ in (\ref{cip}) holds. To see the convergence  in (\ref{cip}) for ${\rm II}$, we simply use the fact that by (\hyperlink{P1}{P1}) and the definition of $\1_{\{1\}}(Q^{(n)})$, 
\begin{align}\label{1Q}
\langle \delta_x,\1_{\{1\}}(Q^{(n)})\delta_x\rangle\equiv 1/N_n\quad\mbox{on }\Lambda_n
\end{align}
[see  (\ref{ip}) and Section~\ref{sec:appendix}]. For III, it follows from functional calculus  and  (\hyperlink{P3}{P3})  that on the event $\Lambda_n$, 
\[
{\rm III}\leq \sum_{x\in E_n}\E^{(n)}\big[\e^{-\lambda \M/N_n};X^U_M=x\big]\sum_{\ell =L+1}^\infty \frac{2^\ell \langle \delta_x,(1- {\bf g})^\ell\Id \delta_x\rangle}{(\lambda/N_n+2)^{\ell +1}}\leq \sum_{\ell=L+1}^\infty (1-\mathbf g)^\ell
\]
so that $\mathbf P({\rm III}>\vep/4,\Gamma_n)=0$ by the first inequality in (\ref{def:L}). The second inequality in (\ref{def:L}) gives $\mathbf P({\rm III}'>\vep/4,\Gamma_n)= 0$. We have proved all the equalities in (\ref{cip}). The proof is complete.
\end{proof}

The main result of this note is the following theorem. For its statement, we write $\ms L(\xi)$ for the law of a random variable $\xi$, and the convergence in \eqref{conv:M} refers to convergence in distribution as $j\to\infty$. Recall that $\mathbf e$ denotes an exponential random variable with $\E[\mathbf e]=1$.

\begin{thm}[Main result]\label{thm:MUU}
\sl   Fix an integer $k\geq 3$, and consider the random walks on random regular $k$-graphs $\{G_n\}$, where the graphs are subject to probability $\mathbf P$. Every subsequence $\{G_{n_i}\}$ contains a further subsequence $\{G_{n_{i_j}}\}$ such that 
\begin{align}\label{conv:M}
\ms L\left(\frac{\M}{N_{n_{i_j}}}\right)\xrightarrow[j\to\infty]{({\rm d})}\ms L\left(\frac{1}{2}\left(\frac{k-1}{k-2}\right)\mathbf e\right)\quad\mbox{$\mathbf P$-a.s.,}
\end{align}
where the law of $M/N_n$ is understood to be under $\P^{(n)}$ for every $n$. Also,
\begin{align}\label{cnst:M}
\E^{(n)}\left[\left(\frac{M}{N_n}\right)^\ell\right]\xrightarrow[n\to\infty]{{\mathbf P}}\ell!\left[\frac{1}{2}\left(\frac{k-1}{k-2}\right)\right]^\ell,\quad \forall\;\ell\in \Bbb  N.
\end{align}
\end{thm}
\begin{proof}
It follows from (\ref{G-rec}) that on $\Lambda_n$,
\begin{align}
\forall\;\lambda_0\in (0,\infty),\quad \sum_{x\in E_n}\pi^{(n)}(x)^2&=\lambda_0 \E^{(n)}[\e^{-\lambda_0 M}G_{\lambda_0}(X^U_M,X^U_M)]\label{res-id0}\\
&=\frac{1}{N_n} \E^{(n)}[\e^{-\lambda_0 M}]+\lambda_0 \E^{(n)}[\e^{-\lambda_0 M}G^<_{\lambda_0}(X^U_M,X^U_M)].\label{res-id}
\end{align} 
Here, the second equality follows from \eqref{Gxy0}, the definition of $G^<_{\lambda_0}$ in \eqref{def:G<}, and \eqref{1Q} since
\[
\left\langle \delta_x,\frac{\1_{\{1\}}(Q^{(n)})}{\lambda_0+2(1-Q^{(n)})}\delta_x\right\rangle=\frac{1}{\lambda_0}\big\langle \delta_x,\1_{\{1\}}(Q^{(n)})\delta_x\big\rangle=\frac{1}{\lambda_0N_n},\quad \forall\;x\in E_n.
\]
The sum of squares on the left-hand side of \eqref{res-id0} is equal to $1/N_n$ by (\hyperlink{P1}{P1}). Therefore, for all $\lambda\in (0,\infty)$, applying Lemma~\ref{lem:main} and  (\ref{sum:Qell}) to \eqref{res-id} with $\lambda_0=\lambda/N_n$ leads to
\[
\E^{(n)}[\e^{-\lambda M/N_n}]\xrightarrow[n\to\infty]{\bf P}\frac{1}{1+\frac{\lambda}{2}\sum_{\ell=0}^\infty Q^{(\infty),\ell}}=\frac{1}{1+\frac{\lambda}{2}\big(\frac{k-1}{k-2}\big)}=\E\big[\e^{-\frac{\lambda}{2}\left(\frac{k-1}{k-2}\right)\mathbf e}\big].
\]
By monotonicity and Cantor's diagonalization, we can find a subsequence $\{G_{n_{i_j}}\}$ such that the foregoing convergences hold for all $\lambda \in[0,\infty)$ $\mathbf P$-a.s. This proves \eqref{conv:M}. For the proof of \eqref{cnst:M}, we further require that $\mathbf  P(\liminf \Lambda_{n_{i_j}})=1$. See Remark~\ref{rmk:rrg} (1). 

To obtain (\ref{cnst:M}), we first note that  $\E^{(n)}[M^\ell]<\infty$ on $\Lambda_n$ for all $\ell\in \Bbb  N$ since the finite chain $(X,Y)$ is irreducible so that $X$ and $Y$ meet at an exponential rate \cite[Proposition~6.3 in Chapter~I]{Asmussen}. Now, we differentiate both sides of (\ref{res-id}) at $\lambda_0=0$ $\ell$ times and get
\begin{align}\label{mom}
0&=\frac{(-1)^\ell}{N_n}\E^{(n)}[M^\ell]+\sum_{r=0}^{\ell-1} {\ell-1 \choose r} \E^{(n)}\left[(-M_0)^{\ell-1-r}\frac{\d^r}{\d \lambda_0^{r}}G^<_{\lambda_0}(X^U_M,X^U_M)\big |_{\lambda_0=0}\right].
\end{align}
Given $\frac{\d^r}{\d \lambda_0^r}\frac{1}{(\lambda_0+A)}=\frac{(-1)^r r!}{(\lambda_0+A)^{r+1}}$, it follows from (\hyperlink{P3}{P3}) and functional calculus that the $r$-th derivatives of $G^<_{\lambda_0}(x,x)$'s are bounded by $r!/(2\mathbf g)^{r+1}$ on $\Lambda_n$. Hence,  for $\{G_{n_{i_j}}\}$ chosen above, \eqref{mom} and induction imply that for any $\ell$, $\{M^\ell/N_{n_{i_j}}^\ell\}$ is uniformly integrable on $\liminf \Lambda_{n_{i_j}}$. We deduce the limits in (\ref{cnst:M}) from the standard formulas of moments of $\mathbf e$.
\end{proof}

\begin{rmk}\label{rmk:mode}
(1) Thanks to  dominated convergence and the uniform integrability observed in the proof, the mode of convergence in (\ref{conv:M}) can be reinforced to convergence in the $L_1$-Wasserstein metric for probability measures. Indeed, on the real line, it is known \cite{Vallender_1974} that the metric can be represented as the $L_1$-norm of the differences of tail distributions. \smallskip 

\noindent (2) See \cite{Pittel} for results that obtain the explicit limit of hitting times of certain Markov chains on infinite sets. \mbox{}\hfill $\blacksquare$
\end{rmk}

As pointed out in Section~\ref{sec:intro}, Theorem~\ref{thm:MUU} obtains the explicit asymptotics of the meeting times. We stress again that it is not a convergence implied by an exponential approximation of the normalized times $M/\E[M]$ as in \cite[Theorem~1.4]{Aldous:AE}, \cite[Theorem~1]{AB} and \cite[Proposition~3.23]{AF}. The methods of proof are also different. After all, our interest is focused on proving \eqref{rrg} and the possible extensions.
 
To see the differences in the proofs, recall that the present method extends the Green function expansion \eqref{norm_trace}. The approximations aim to use only the spectra of the Markov chains (rather than the product chains). In contrast, the method in \cite{Aldous:AE} considers a perturbation-type extension of the following identity, among other things:  The hitting time of a set by an irreducible finite Markov chain (subject to mild conditions) is precisely exponential if the chain starts from the quasi-stationary limit distribution for the first hit. This exponential distribution has an implicit parameter given by the first moment of the hitting time; this moment can be characterized analytically by a variational problem \cite[Section~3.6.5]{AF}.  

The methods in \cite{AB} and \cite[Sections~3.5]{AF} are also different from the present method. They apply the following inequality due to Brown \cite{Brown}: If the tail distribution of a nonnegative random variable $T$ is completely monotone, that is, if the tail distribution is the Laplace transform of a measure, then
\begin{align}
\sup_{t\geq 0}|\P(T>t)-\e^{-t/\E[T]}|\leq \frac{\E[T^2]}{2\E[T]^2}-1.\label{Brown}
\end{align}
For hitting times of sets, the proofs in \cite{AB,AF} apply \eqref{Brown} by working with the reduced transition kernels that collapse the sets to be hit to singletons. Spectral representations in the eigenvalues and eigenfunctions of the reduced kernels are derived for comparing the first two moments as those in the bound of \eqref{Brown}. Moreover, these methods yield bounds for the expected hitting times \cite[Lemma~2]{AB} and \cite[Section~3.5.3]{AF}. These bounds are elementary expressions in the spectral gaps, stationary distributions and transition kernels in one step. Nevertheless, Theorem~\ref{thm:MUU} calls for the use of eigenvalue \emph{distributions}, hence essentially almost most of the eigenvalues, of the transition kernels. 

\section{Probability distributions at meeting times}\label{sec:prob}
In the rest of this note, we resume the general setup that $Q$ is irreducible and reversible and is defined on a finite nonempty set $E$. 

The following proposition shows the domination $\E[\e^{-\lambda M};X^U_M=x]\leq {\rm cnst}(\lambda)\cdot \pi(x)^2$ that is used in (\ref{I-I}). The proof naturally extends to a linear equation \eqref{linear_eq} satisfied by  $x\mapsto \pi(x)^{-1}\E[\e^{-\lambda M};X^U_M=x]$; whether it already appears in the literature is not known to us.

\begin{prop}\label{prop:recursion}
\sl For any $\lambda\in (0,\infty)$, we have
\begin{align}\label{dominate}
\E\big[\e^{-\lambda M};X^U_\M=x\big]\leq \frac{2+\lambda}{\lambda}\pi(x)^2,\quad \forall\;x\in E.
\end{align}
Moreover,  the function
\begin{align}\label{def:F}
F_\lambda(x)=\pi(x)^{-1}\E\big[\e^{-\lambda M};X^U_{M}=x\big],\quad x\in E,
\end{align}
solves the following linear equation:
\begin{align}\label{linear_eq}
(\Id+R_\lambda)F_\lambda=\frac{2+\lambda}{\lambda}\pi,
\end{align}
where $R_\lambda$ is a symmetric matrix with nonnegative entries defined by
\begin{align}\label{def:Slambda}
R_\lambda(x,y)=2\frac{\pi(y)}{\pi(x)}\int_0^\infty \e^{-\lambda t} Q_t(y,x) QQ_t(y,x)\d t,\quad x,y\in E.
\end{align}
\end{prop}
\begin{proof}
Write $\F_t=\sigma(X_s,Y_s;s\leq t)$ and $M+J$ for the first update time of $(X^U,Y^{U'})$ after $M$. Since $X$ and $Y$ are independent rate-$1$ chains, $\P(J>t|\F_M)= \P(\mathbf e_1\wedge \mathbf e_2>t)=\e^{-2t}$ for independent exponential variables $\mathbf e_1$ and $\mathbf e_2$ with mean $1$. Here, $\mathbf e_1\wedge \mathbf e_2$ denotes the minimum of $\mathbf e_1$ and $\mathbf e_2$.

Now, by stationarity and the independence of $X^U$ and $Y^{U'}$, we have: for all $x\in E$,
\begin{align}
 \pi(x)^2&=\int_0^\infty \lambda \e^{-\lambda t}\P(X^U_t=x,Y^{U'}_t=x)\d t\notag\\
\begin{split}
&= \E\left[\int_M^{M+J}\lambda \e^{-\lambda t}\1_{\{X^U_t=x,Y^{U'}_t=x\}}\d t\right]  +\E\left[\int_{M+J}^\infty \lambda \e^{-\lambda t}\1_{\{X^U_t=x,Y^{U'}_t=x\}}\d t\right].\label{recursion:1}
\end{split}
\end{align}
We compute the two terms in the last equality separately. First, since $X^U_t=Y^{U'}_t$ over $t\in [M,M+J)$, the first term in (\ref{recursion:1}) satisfies 
\begin{align}
\E\left[\int_M^{M+J}\lambda \e^{-\lambda t}\1_{\{X^U_t=x,Y^{U'}_t=x\}}\d t\right]
&=\E\left[\e^{-\lambda M}-\e^{-\lambda(M+J)};X^U_M=x\right]\notag\\
&=\E\left[\e^{-\lambda M}\1_{\{X^U_M=x\}}\E\big[1-\e^{-\lambda (\mathbf e_1\wedge \mathbf e_2)}\big]\right]\notag\\
&=\E\left[\e^{-\lambda M};X^U_M=x\right]\frac{\lambda }{2+\lambda}.
\label{recursion:3}
\end{align}
Here, the second equality follows from the strong Markov property at $M$ and the above mentioned conditional distribution of $J$. The last equality and \eqref{recursion:1} prove (\ref{dominate}).  

To obtain (\ref{linear_eq}), we compute the second term in (\ref{recursion:1}). The strong Markov property at time $M+J$ and the independence of $X$ and $Y$ give 
\begin{align*}
&\E\left[\int_{M+J}^\infty \lambda \e^{-\lambda t}\1_{\{X^U_t=x,Y^{U'}_t=x\}}\d t\right]\\
=&\sum_{a,b\in E} \E\left[\lambda \e^{-\lambda (M+J)};X^U_{M+J}=a,Y^{U'}_{M+J}=b\right] \int_0^\infty \e^{-\lambda t} Q_t(a,x)Q_t(b,x)\d t.
\end{align*}
For any $a,b\in E$, we apply the strong Markov property of $(X^U,Y^{U'})$ at $M$ to the last expectation. Then to evaluate $\E[\e^{-\lambda J};X^U_{M+J}=a,Y^{U'}_{M+J}=b|\F_M]$, note that conditioned on $\F_M$, one of the two coordinates of $(X_{M+t}^U,Y^{U'}_{M+t};t\geq 0)$ jumps at time $J$ according to $Q(X^U_M,\,\cdot\,)$ with equal probability. Also, as a basic property of Markov chains, the jump is conditionally independent of $J$, while $J$ is conditionally distributed as $\mathbf e_1\wedge \mathbf e_2$. We get
\begin{align*}
&\quad \E\left[\lambda \e^{-\lambda (M+J)};X^U_{M+J}=a,Y^{U'}_{M+J}=b\right]\\
&=\E\left[\e^{-\lambda M}\E\big[\lambda \e^{-\lambda (\mathbf e_1\wedge \mathbf e_2)}\big];X^U_M=a\right]\frac{1}{2}Q(a,b)+\E\left[\e^{-\lambda M}\E\big[\lambda \e^{-\lambda (\mathbf e_1\wedge \mathbf e_2)}\big];Y^{U'}_M=b\right]\frac{1}{2}Q(b,a)\\
&=\E\left[\e^{-\lambda M}\left(\frac{2\lambda}{2+\lambda}\right);X^U_M=a\right]\frac{1}{2}Q(a,b)+\E\left[\e^{-\lambda M}\left(\frac{2\lambda }{2+\lambda}\right);X^U_M=b\right]\frac{1}{2}Q(b,a),
\end{align*}
where we use $X^U_M=Y^{U'}_M$. Putting the last two displays together, we get
\begin{align}
&\quad \E\left[\int_{M+J}^\infty \lambda \e^{-\lambda t}\1_{\{X^U_t=x,Y^{U'}_t=x\}}\d t\right]\notag\\
&=\frac{2\lambda}{2+\lambda}\sum_{a,b\in E}\E\left[\e^{-\lambda M};X^U_M=a\right]Q(a,b)\int_0^\infty  \e^{-\lambda t}Q_t(a,x)Q_t(b,x)\d t\notag\\
&=\frac{2\lambda}{2+\lambda}\sum_{a\in E}\E\left[\e^{-\lambda M};X^U_M=a\right]\int_0^\infty  \e^{-\lambda t}Q_t(a,x)QQ_t(a,x)\d t.
\label{recursion:4}
\end{align}

Finally, we apply \eqref{recursion:3} and \eqref{recursion:4} to \eqref{recursion:1} and get
\begin{align}
\begin{split}\label{eq:Flambda}
\pi(x)^2&=\E\left[\e^{-\lambda M};X^U_M=x\right]\frac{\lambda }{2+\lambda}\\
&\quad +\frac{2\lambda}{2+\lambda}\sum_{a\in E}\E\left[\e^{-\lambda M};X^U_M=a\right]\int_0^\infty  \e^{-\lambda t}Q_t(a,x)QQ_t(a,x)\d t.
\end{split}
\end{align}
The required equality in (\ref{linear_eq}) now follows by dividing both sides of the last equality  by $\tfrac{\lambda \pi(x)}{2+\lambda}$ and applying the definitions \eqref{def:F}  and \eqref{def:Slambda} of $F_\lambda$ and $R_\lambda$. (The entries $\pi(x)$'s are strictly positive by the irreducibility of $Q$ \cite[Proposition~1.14]{LPW}.) The proof is complete.  
\end{proof}

\section{Appendix: Functional calculus for transition kernels}\label{sec:appendix}
\label{sec:factorization} 
In this section, we recall the functional calculus for transition kernels $Q$ subject to the general assumptions at the beginning of Section~\ref{sec:prob}.

The functional calculus is based on the following transform of $Q$:
\begin{align}\label{def:A}
S(x,y)\;\defeq\;\pi(x)^{1/2}Q(x,y)\pi(y)^{-1/2},\quad x,y\in E.\textcolor{black}{} 
\end{align}
This matrix $S$ is symmetric and its spectrum $\sigma(S)$ is the same as the spectrum $\sigma(Q)$ of $Q$. Our notation here is that an eigenvalue $q$ of multiplicity $m$ appears $m$ times in $\sigma(S)$ and in $\sigma(Q)$. In particular, $1$ is an eigenvalue and has multiplicity $1$. In addition, $S$ admits a set $\{\varphi_q;q\in \sigma(S)\}$ of real-valued eigenvectors,  orthonormal with respect to the inner product defined in (\ref{ip}), such that $\pi^{1/2}$ is the eigenvector associated with the eigenvalue $1$. See \cite[Section~12.1]{LPW} for details of these properties of $S$.

Given the above setup, the matrices $f(S)$ for functions $f:[-1,1]\to \Bbb  C$ are defined by 
\begin{align}\label{def:fA}
f(S)(x,y)\;\defeq\;\sum_{q\in \sigma(S)}f(q)\varphi_q(x)\varphi_{q}(y).
\end{align}
This class of matrices is an extension of $\{S^\ell;\ell\in \Bbb  Z_+\}$ so that for all $f_n,f:[-1,1]\to \Bbb  C$ and $a,b\in \Bbb  C$,  (i) $f_n\to f$ pointwise implies $f_n(S)\to f(S)$ entrywise, (ii)  $(af+bg)(S)=af(S)+bg(S)$ and $fg(S)=f(S)g(S)$. 

With $S$ replaced by $Q$, the above properties apply to the matrices $f(Q)$ defined by 
\begin{align}\label{def:fQ}
f(Q)(x,y)=\pi(x)^{-1/2} f(S)(x,y)\pi(y)^{1/2},\quad x,y\in E.
\end{align}
For example, the transition kernels $Q_t$ of the rate-1 $Q$-Markov chain can be written as 
\begin{align}\label{Qt:sum}
Q_t(x,y)&\,\defeq\;\sum_{\ell=0}^\infty \e^{-t}t^\ell\frac{Q^\ell(x,y)}{\ell!}
=\pi(x)^{-1/2}\sum_{\ell=0}^\infty \e^{-t}t^\ell\frac{S^\ell(x,y)}{\ell!}\pi(y)^{1/2}\notag\\
&=\pi(x)^{-1/2}\e^{t(S-1)}(x,y)\pi(y)^{1/2}=\e^{t(Q-1)}(x,y),
\end{align}
where the last three equalities follow from (\ref{def:A}), (\ref{def:fA}) and (\ref{def:fQ}), respectively. \\

\noindent {\bf Acknowledgements.} The author would like to thank an Associate Editor for suggesting to review earlier methods for the exponential approximations of hitting times and a referee for the careful reading of the manuscript.

\end{document}